\documentclass[12pt]{amsart}

\usepackage{amsmath} 
\usepackage{amsfonts}
\usepackage{amssymb}
\usepackage{amsthm}

\usepackage{latexsym}
\usepackage{color}
\usepackage{enumerate}
\usepackage{mathrsfs}
\usepackage{comment}
\usepackage{setspace}
\usepackage{ulem}

\usepackage{appendix}

\usepackage{mathtools}
\usepackage{hyperref}

\newtheorem{lemma}{Lemma}[section]
\newtheorem{thm}{Theorem}[section]
\newtheorem{prop}{Proposition}[section]
\newtheorem{cor}{Corollary}[section]

\theoremstyle{remark}
\newtheorem{remark}{Remark}[section]

\numberwithin{equation}{section}

\def\tr{\textmd{tr}}

\def\Ric{\textmd{Ric}}

\def\R{\mathbb{R}}

\def\R{\mathbb{R}}

\def\p{\partial}

\def\m{\mathfrak{m}}
\def\mh{\mathfrak{m}_{_H}}

\def\({\left(}
\def\){\right)}
\def\[{\left[}
\def\]{\right]}

\def\Sh{\Sigma_{h}}

\def\mby{\m_{_{BY}}}
\def\mb{\m_{_B}}
\def\gamm{\gamma^{(m)}}
\def\umi{ u_{m_i}^{(i)}}

\def\mhh{\mathfrak{m}^{_\mathbb{H}}_{_H}}

\def\mbh{\mathfrak{m}^{_\mathbb{H}}_{_B}}

\newcommand{\be}{\begin{equation}}
\newcommand{\ee}{\end{equation}}
\newcommand{\bee}{\begin{equation*}}
\newcommand{\eee}{\end{equation*}}

\begin{document}

\title{On Hawking mass and Bartnik mass of CMC surfaces}

\author{Pengzi Miao}
\address[Pengzi Miao]{Department of Mathematics, University of Miami, Coral Gables, FL 33146, USA}
\email{pengzim@math.miami.edu}

\author{Yaohua Wang}
\address[Yaohua Wang]{School of Mathematics and Statistics, Henan University,
Kaifeng, Henan 475004, China.}
\email{wangyaohua@henu.edu.cn}

\author{Naqing Xie}
\address[Naqing Xie]{School of Mathematical Sciences, Fudan
University, Shanghai 200433, China.}
\email{nqxie@fudan.edu.cn}




\begin{abstract}
Given a constant mean curvature surface that bounds a compact manifold with nonnegative scalar curvature,
we obtain intrinsic conditions on the surface  that guarantee the positivity of its Hawking mass. 
We also obtain estimates of the Bartnik mass of such surfaces, without assumptions on 
the integral of the squared  mean curvature. If  the ambient manifold has negative scalar curvature, 
our method also applies and yields estimates on the hyperbolic Bartnik mass of these surfaces. 
\end{abstract}


\maketitle

\markboth{Pengzi Miao, Yaohua Wang and Naqing Xie}{On Hawking mass and Bartnik mass of CMC surfaces}

\section{Introduction}
Given a Riemannian $3$-manifold $M$, let $ \Sigma \subset M $ be a closed $2$-surface with a unit normal vector filed 
$\nu$. $\Sigma$ is called a CMC  surface if its mean curvature  with respect to $\nu$ is a constant. Throughout this paper, 
we assume $\Sigma$ is a CMC surface that is topologically a sphere. 

When the ambient manifold $M$ has nonnegative scalar curvature, a classic result of Christodoulou and Yau \cite{C-Y}  is the following:
\begin{thm}[\cite{C-Y}]
Suppose $ \Sigma$ is a stable, $CMC$ sphere in a $3$-manifold $M$ with nonnegative scalar curvature, then
$ \mh (\Sigma) \ge 0 $. 
\end{thm}
Here $ \mh (\Sigma)$ is the Hawking quasi-local mass \cite{Hawking} of $ \Sigma$ in $M$,  given by
\be
\mh (\Sigma) = \sqrt{ \frac{ | \Sigma |}{16 \pi} } \left( 1 - \frac{1}{16 \pi} \int_\Sigma H^2 \, d \sigma \right) ,
\ee
where $ |\Sigma | $ is the area and $ H$ is the mean curvature of $ \Sigma$, respectively, and 
$d \sigma$ denotes the area form on $ \Sigma$.
A CMC surface $ \Sigma$ is called stable if 
\be \label{eq-stable}
\int_\Sigma | \nabla f |^2 - ( | A|^2 + \Ric (\nu, \nu) ) f^2 \, d \sigma \ge 0 
\ee
for any function $f$ on $ \Sigma$ with $ \int_\Sigma f \, d \sigma = 0 $, where $ \nabla $ denotes  the gradient on $ \Sigma$, 
$ A$ is the second fundamental  form of $ \Sigma$ and $ \Ric (\nu, \nu) $ is the Ricci curvature of $M$ along $ \nu$.

The stability condition \eqref{eq-stable} is a natural geometric condition and it plays a key role in the 
estimate of $ \mh(\Sigma)$ in \cite{C-Y}.

In this paper,  one of the main questions that we consider is the non-negativity of $\mh (\Sigma)$ 
without imposing the stability condition on $ \Sigma$.
Instead, we assume $ \Sigma$ bounds a finite region $\Omega$ with nonnegative scalar curvature. 
There are two reasons for making such  a consideration:

\begin{enumerate}
\item  [i)] First, from a quasi-local mass point of view, it is desirable to draw information on the quasi-local mass of $\Sigma$ 
purely from knowledge on the geometric data $(g, H)$, where $g$ is the intrinsic metric on $\Sigma$ and $H$ is the mean curvature; 

\item [ii)] Second, in the special case when $g$ is a round metric on $\Sigma$, one indeed knows
$$ \mh (\Sigma) \ge 0 $$
for any CMC surface $ \Sigma$ with positive constant mean curvature $H_o$. 
 This is a consequence of the Riemannian positive mass theorem
\cite{Schoen-Yau79, Witten81}. To see this, suppose $ \Sigma = \p \Omega$ where $ \Omega$ is  compact 
and has nonnegative scalar curvature. Gluing $ \Omega$ with an exterior Euclidean region $ \R^3 \setminus B$, where
$ B$ is a round ball with  boundary $\p B$ isometric to $ \Sigma$, one concludes $ H_o \le H_{_E} $,
where $H_{_E}$ is the constant mean curvature of $ \p B $ in $ \R^3$ (see \cite{Miao02, ShiTam02}).
As a result, $ \mh (\Sigma) \ge 0$.
\end{enumerate}

\vspace{.1cm}

In relation to ii) above, it is natural to ask  if $ \mh (\Sigma) $  has positivity property when the intrinsic metric 
on $\Sigma$ is not far from being round. As an application of our main result, Theorem \ref{thm-main} stated in a moment, 
we establish positivity of $\mh (\Sigma)$ for these surfaces.  

To formulate our theorems, we make use of 
a scaling invariant number $\zeta_g$ that measures 
how far a metric $g$ is from a round metric.  This $\zeta_g$ was introduced in \cite{Miao-Xie16} and we recall it here. 
Given any metric $g$  with  positive Gauss curvature $K_g  $ on the sphere $ S^2$, 
let $ r_o $ be the area radius of $(S^2, g)$, i.e., $ | S^2 |_g = 4 \pi r_o^2 $. 
Let  $\{ g(t) \}_{0 \leq t \leq 1}$ be a smooth path of metrics  on $S^2$ such that 
$g(0)=g$, $g(1)$ is round,  $g(t)$ has positive Gauss curvature $K_{g(t)}$ and 
$ \tr_{g(t)} g'(t)=0$ for all $t$.
(Existence of such a path, for instance,  follows from Mantoulidis and Schoen's proof of \cite[Lemma 1.2]{M-S}.)
Associated to this path  $\{ g(t) \}_{0 \leq t \leq 1}$,  let $\alpha $ and $\beta$ be two  constants given by 
\begin{equation} \label{eq-alpha-beta}
\begin{split}
\alpha =\dfrac{1}{4} \max \limits_{t\in[0,1]}  \max\limits_{ S^2 }  |g'(t)|^2_{g(t)}, \  
\beta  = r_o^2   \min \limits_{t\in[0,1]}   \min\limits_{ S^2 } K_{g(t)}.
\end{split}
\end{equation}
It is clear $ \beta \in (0, 1] $ by the Gauss-Bonnet Theorem, and $ \alpha > 0 $ if  $g$ is not a  round metric. 
With these notations, we let 
\begin{equation} \label{eq-df-zeta}
\zeta_g  =  \inf_{ \{ g(t) \} } \left( \frac{\alpha}{ 2 \beta } \right)^\frac12 ,
\end{equation}
where the infimum is taken over all such paths $ \{ g(t) \}_{ 0 \le  t \le 1}$. 
We point out that  $ \zeta_g $ in \eqref{eq-df-zeta} satisfies   $ 2 \zeta_g^2  = \eta(g)^{-1}$, 
where $\eta(g)$ was defined in \cite[Section 4]{Miao-Xie16}. 

Evidently, $\zeta_g = 0 $ if $g$ is a round metric; moreover,  $\zeta_g$ is invariant under constant scaling of $g$.
For any  $ \gamma \in (0,1)$, it was shown in  \cite[Proposition 4.1]{Miao-Xie16} that, 
if  $ g $ is  $C^{2, \gamma}$-close to a round metric  $g_*$, normalized with area $4\pi$, then 
$ \zeta_g \le  {C} || g - g_* ||_{C^{0, \gamma} (\Sigma) }  $ where
$C$ is an absolute constant.

The following theorem gives a sufficient condition on the intrinsic metric on $\Sigma$ that guarantees 
the positivity of $ \mh(\Sigma)$.

\begin{thm} \label{thm-positive-mH}
Let $ M $ be a Riemannian $3$-manifold with nonnegative scalar curvature, with 
boundary $\p M$, which is a minimal surface (possibly disconnected) minimizing area among all closed surfaces 
which bound a domain with $ \p M $. 
Suppose $ \Sigma \subset M$ is a CMC surface  bounding a  domain $\Omega$ with $ \p M $
and $\Sigma$ has positive mean curvature  with respect to the unit normal pointing out of $\Omega$.
Let $ g$ be the intrinsic metric on $ \Sigma$.  Suppose $ g$ has positive Gauss curvature. 
If 
$$ \zeta_g <  C  \sqrt{ \frac{ | \p M | }{ | \Sigma | } } , $$
then 
$ \mh (\Sigma) > 0 . $
Here $C$ is some absolute constant (for instance $C$ can be $\frac{\sqrt{2}}{3}$).
\end{thm}

\begin{remark}
A manifold $M$ in Theorem \ref{thm-positive-mH} can be taken as  an asymptotically flat $3$-manifold 
for which the Riemannian Penrose inequality \cite{Bray01, H-I01} applies.
\end{remark}

We will deduce Theorem \ref{thm-positive-mH} from a general result which holds without 
assumptions on $\zeta_g$.

\begin{thm} \label{thm-main}
Suppose $ \Sigma$ is a CMC surface that bounds a compact $3$-manifold $\Omega$ with nonnegative scalar curvature, which 
may have  nonempty interior horizon. Precisely, this means that $ \Sigma$ is a boundary component of $ \p \Omega$
and $ \Sh : = \p \Omega \setminus \Sigma$, if nonempty, is a minimal surface that minimizes area among surfaces enclosing 
$ \Sh$. 
Suppose the intrinsic metric $g$ on $\Sigma$ has positive Gauss curvature and 
the mean curvature of $ \Sigma$ with respect to the outward normal $\nu$ is a positive constant $H_o$. 
Let $r_o = \sqrt{ \frac{ | \Sigma|}{ 4 \pi } } $ and define $ \tau = \frac12 r_o H_o $. 
Let $ \theta $ be the unique root to 
\be \label{eq-theta-1}
\theta^3 -  \frac{3 \zeta_g \tau }{2}  \theta^2  - 1 = 0 . 
\ee
Then the following holds:
\begin{enumerate}
\item [a)] If $ \Sh = \emptyset $, i.e. $ \Sigma = \p \Omega$, then  
\bee
 \tau \le \theta .
\eee
\item[b)] If $ \Sh \neq \emptyset $, then 
$$ 
\tau^2 + \frac{r_h}{r_o}  \le \theta^2 .
$$
Here $ r_h = \sqrt{ \frac{ | \Sh |}{4\pi} }$. 
\item[c)] Let $\mb(\Sigma)$ denote the Bartnik quasi-local mass of $ \Sigma$, then
\bee
\begin{split}
\mb (\Sigma) \le & \ \sqrt{ \frac{ | \Sigma |}{16 \pi} } \left( \theta^2 - 1 \right)   + \mh (\Sigma) \\
= & \ \sqrt{ \frac{ | \Sigma |}{16 \pi} } \left( \theta^2 - \tau^2 \right)  .
\end{split}
\eee
In particular, this shows 
$ \mb (\Sigma) \le   C r_o  
\left( 1 +   \zeta_g \tau   \right)  \zeta_g \tau  + \mh (\Sigma) $,
where $C$ is an absolute constant. 

\end{enumerate}
\end{thm}

We defer the definition of the Bartnik mass $\mb(\cdot)$ to  the next section.
For the moment,  we give a few remarks about Theorem \ref{thm-main}.

\begin{remark}
The constant  $\tau$ satisfies $ \tau^2 = \frac{1}{16 \pi} \int_{\Sigma} H_o^2 \, d \sigma $.
Thus, $ \mh (\Sigma) > 0 \Leftrightarrow \tau < 1 $.
In terms of $\mh(\Sigma)$, a) and b) of Theorem \ref{thm-main} can be rewritten as
\be \label{eq-mh-bound}
\mh(\Sigma) 
\ge \left\{
\begin{array}{lc}
 \frac{r_o}{2} ( 1 - \theta^2)  ,  & \ \ \mathrm{if} \ \Sh = \emptyset; \\
 \ & \ \\
 \frac{r_o}{2} \left( 1 + \frac{r_h}{r_o} - \theta^2 \right), & \ \ \mathrm{if} \ \Sh \neq \emptyset .
\end{array} 
 \right.
\ee
Similarly, c) of Theorem \ref{thm-main} can be rewritten 
as
\be
\mb(\Sigma) \le 
\left\{
\begin{array}{lc}
\frac{r_o}{2} \left( \theta^2 - 1 \right), & \ \ \mathrm{if} \ \mh(\Sigma) = 0; \\
\ & \ \\
\frac{\theta^2 - \tau^2}{ 1 - \tau^2} \, \mh(\Sigma) , & \ \ \mathrm{if} \ \mh(\Sigma) \neq 0.
\end{array}
\right.
\ee
\end{remark}

\begin{remark} \label{rem-round}
If $ g$ is a round metric, then $ \zeta_g = 0 $ and hence $\theta =1$.
In this case, it is easily  seen Theorem \ref{thm-main} 
is true. For instance, a) follows from ii) above; b) is a special case of the result in \cite{Miao09}; and c) follows from
the fact that one can attach a spatial Schwarzschild manifold with mass $\m = \mh (\Sigma)$ to $ \Omega$ at $ \Sigma$.
\end{remark}

\begin{remark}
Conclusions in a) and   b) of Theorem \ref{thm-main} concern how nonnegative scalar curvature and interior horizon affect $ \mh(\Sigma)$  for a CMC surface.  This question was studied by the first and the third authors in \cite{Miao-Xie16}. Under smallness assumptions on $\tau$, results weaker than a) and b) were derived  in \cite{Miao-Xie16}.

An upper bound of $\mb(\Sigma)$ for CMC surfaces was first  derived by  Lin and Sormani \cite{L-S}
for an arbitrary  metric $g$ on $\Sigma$. If $H_o = 0 $ and  the first eigenvalue of $ - \Delta_g + K_g$ is positive, 
 Mantoulidis and Schoen \cite{M-S} proved  $ \mb (\Sigma) = \mh (\Sigma) $.
Assuming $K_g > 0 $ and imposing the smallness assumption on $\tau$  used in \cite{Miao-Xie16},  
an upper bound of $\mb (\Sigma)$ was derived  by Cabrera Pacheco, Cederbaum, McCormick and 
the first author \cite{CCMM}. A comparison of the  estimates in \cite{L-S} and \cite{CCMM} can be found in \cite[Remark 1.2]{CCMM}. 
Our estimate of $\mb(\Sigma)$ in c) of Theorem \ref{thm-main} shares the same feature as that in \cite{CCMM}, 
but holds without assumptions on $\tau$. 
 \end{remark}

\begin{remark} \label{rem-Bmass-mh}
If one does not assume $ \Sigma$ bounds a manifold with nonnegative scalar curvature, the estimate 
of $\mb(\Sigma)$ in c) of Theorem \ref{thm-main} is still valid  provided the pair $(g, H_o)$ 
satisfies $\mh(\Sigma) \ge 0 $. See Remark \ref{eq-bmass-no-compact} for detailed reasons. 
\end{remark}

As a corollary of  Remark \ref{rem-Bmass-mh} and the theorem of Christodoulou and Yau, we have

\begin{cor}
The Bartnik mass of any stable CMC surface $\Sigma$ with positive Gauss curvature in a $3$-manifold 
with nonnegative scalar curvature satisfies the estimate in c) of Theorem \ref{thm-main}. 
\end{cor}

We  have an analogue of Theorem \ref{thm-positive-mH} 
with $ \sqrt{  \frac{ | \p M | }{ | \Sigma | }  }$ replaced by $ \frac{ 2\mb(\Sigma)}{r_o}$.

\begin{thm} \label{thm-positive-mH-b}
Let  $ \Sigma$ be a CMC surface, with positive mean curvature $H_o$, bounding a compact $3$-manifold $\Omega$ 
with nonnegative scalar curvature. Suppose $ \mb (\Sigma) > 0 $ and 
 the intrinsic metric $g$ on $\Sigma$ has positive Gauss curvature. 
If
$$
\zeta_g <  C \left(  1 +  \frac{2 \mb(\Sigma)}{r_o}  \right)^{-1}  \min \left\{ \frac{2 \mb(\Sigma)}{r_o} , 1 \right\}  ,
$$
then $ \mh (\Sigma) > 0 $.
Here $ r_o = \sqrt{ \frac{ | \Sigma|}{4 \pi} } $ and 
$C$ is some absolute constant (for instance $C$ can be $\frac{\sqrt{2}}{3}$).
\end{thm}

\begin{remark}
In the setting of Theorem \ref{thm-positive-mH-b}, one may also consider 
the Brown-York mass  of $\Sigma$ \cite{BY1, BY2},  given by 
$ \mby (\Sigma) = \frac{1}{ 8 \pi} \int_{\Sigma} (H_{_E} - H_o ) \, d \sigma $, 
where $ H_{_E}$ is the mean curvature of the isometric embedding of $ (\Sigma, g)$ in $ \R^3$.
As $H_o$ is a constant, one has 
\be
 \mby (\Sigma) =  \m_{_H}(\Sigma)  +  \left( \frac{1}{ 8 \pi} \int_{\Sigma} H_{_E} \, d \sigma - r_o  \right)  +  \frac{r_o}{2}  ( 1 -  \tau )^2 ,
\ee
where the second  term in the bracket is nonnegative by the Minkowski inequality.
In \cite{ShiTam02}, Shi and Tam proved $ \mby (\Sigma) \ge 0 $.
It would be interesting to know if the positivity of $ \mby(\Sigma)$ can be used in the study of $ \mh (\Sigma)$. 
\end{remark}

\begin{remark}
In relation to the positivity  of $ \mh (\Sigma)$, a natural question is its rigidity. 
Under the assumption $ \Sigma$ is stable,  recent results concerning $ \mh (\Sigma) = 0 $ 
were given by Sun \cite{Sun17} and by Shi, Sun, Tian and Wei \cite{SSTW17}.
\end{remark}

Our proof of Theorem \ref{thm-main} is built on the previous work 
of the first and the third authors  \cite{Miao-Xie16}.
The techniques we use  to prove Theorem \ref{thm-main} here can also be applied 
to the setting of manifolds with a negative scalar curvature lower bound. 
It is known in the literature  the Hawking mass $\mh (\Sigma)$ has a 
hyperbolic analogue, $\mhh (\Sigma)$ (see \eqref{eq-def-mh-h}).
Recently, Cabrera Pacheco,  Cederbaum and McCormick \cite{CCM} formulated 
a hyperbolic analogue of the Bartnik mass and derived results analogous to those in
\cite{M-S} and \cite{CCMM}.
Combining the techniques in proving Theorem \ref{thm-main} and a gluing tool from \cite{CCM}, 
we obtain estimates of the hyperbolic Bartnik mass, which we denote by $\mbh (\Sigma)$,  
for the boundary of a compact manifold with negative scalar curvature.

\begin{thm} \label{thm-intro-main-h}
Suppose $ \Sigma$ is a CMC surface bounding  a compact $3$-manifold $\Omega$ with  
scalar curvature $  R \ge - 6 \kappa^2$  for some constant $ \kappa > 0$.
Let $ g$ be the intrinsic metric on $ \Sigma$ and suppose 
its  Gauss curvature  satisfies $K_g > - 3 \kappa^2$.
Let $ \tau = \frac12 H_o r_o$, where $ r_o$ is the area radius of $\Sigma$ and 
$H_o$ is the positive constant mean curvature of $ \Sigma$ in $\Omega$.
Then the hyperbolic Bartnik mass $\mbh (\Sigma)$ satisfies 
\be  \label{eq-intro-main-mbh}
\begin{split}
& \  \mbh(\Sigma) - \mhh(\Sigma) \\
\le & \  \frac{r_o}{2} \left[  \kappa^2 r_o^2  \left( 1 +  \frac32  \tau \xi  \right)^2  + 
  \left( 1 +  \frac32  \tau \xi  \right)^\frac23  - \kappa^2 r_o^2 - 1    \right] \\
\le & \ \frac{r_o}{2} \left( 3  \kappa^2 r_o^2  + 1   \right)  \left( 1 +  \frac34  \tau \xi  \right)  \tau \xi .
\end{split}
\ee
Here   $ \xi \ge 0 $ is a constant that is specified as follows.
\begin{enumerate} 

\item[(i)] When  $ \inf_\Sigma K_g \le 0 $, 
$ \xi =  \zeta_{g, \kappa}$,  where $\zeta_{g, \kappa} $ is a constant determined by $g$, given by 
$ \zeta_{g, \kappa} = \inf_{ \{ g(t)  \} }   \left( \frac{ \alpha}{  2 \beta + 6 \kappa^2 r_o^2  } \right)^{\frac12} $.
Here  the infimum is taken over all  paths of metrics $\{ g(t) \}_{ 0 \le t \le 1}$ with 
$g(0)=g$, $g(1)$ is round,  $ K_{g(t)} >  - 3 \kappa^2 $, and $ \tr_{g(t)} g'(t)=0 $, and 
$ \alpha $, $ \beta $ are  two constants defined in \eqref{eq-alpha-beta}.

\vspace{.2cm}

\item[(ii)] When $ \inf_\Sigma K_g > 0 $, 
$\xi $ is a constant given in \eqref{eq-sec-inf-root}. In particular, $ \xi$ satisfies 
$  \xi \le  \zeta_g \theta^2 \le  \zeta_g  \left( 1 +  \frac32  \tau \zeta_g   \right)^2  $.
Here $\zeta_g$ is given in \eqref{eq-df-zeta} and 
$\theta $ is the unique root to $ \theta^3 - \frac32 \tau \zeta_g \theta^2 - 1 = 0 $. 
\end{enumerate}
\end{thm}

The remainder of this paper is organized as follows. 
In Section \ref{sec-pf-main}, we consider manifolds with nonnegative scalar curvature and 
prove Theorem \ref{thm-main}. 
In Section \ref{sec-other}, we apply Theorem \ref{thm-main} to  prove 
Theorems \ref{thm-positive-mH} and \ref{thm-positive-mH-b}.
In Section \ref{sec-hyper}, we consider manifolds with negative scalar curvature and prove Theorem \ref{thm-intro-main-h}.

\vspace{.3cm}

\noindent {\sl Acknowledgements.}  
The work of PM was partially supported by Simons Foundation Collaboration Grant for Mathematicians \#585168.
The work of YW was partially supported by National Natural Science Foundation of China \#11401168, \#11671089.
The work of NX was partially supported by  National Natural Science Foundation of China \#11671089.
The authors would like to thank the anonymous referees for thoroughly reviewing the original manuscript and for giving valuable suggestions.

\section{Manifolds with nonnegative scalar curvature} \label{sec-pf-main}

Let $\Omega$, $ \Sigma$, $ r_o$, $H_o$ and $\tau$ be given in Theorem \ref{thm-main}. 
By Remark \ref{rem-round}, it suffices to assume that  the intrinsic metric  $g$ on $\Sigma$ is not round.
We divide the proof of Theorem \ref{thm-main} into a few steps:

\vspace{.3cm}

\noindent  \underline{Step 1.}
We review the construction of a suitable metric on $N =  [0,1] \times \Sigma $ from \cite{Miao-Xie16}.
Let $ \{ g(t) \}_{t \in [0,1]}$ be any given  smooth path of metrics on $ \Sigma $, satisfying
$g(0)=g$, $g(1)$ is round,  $K_{g(t)} >  0 $  and 
$ \tr_{g(t)} g'(t)=0 $,  $ \forall \, t $. 
Given any parameter  $ m \in (- \infty, \frac{1}{2} r_o )$, consider
 part of a spatial Schwarzschild metric  
\begin{align*}
\gamma_m = \frac{1}{1 - \frac{2 m}{r} } d r^2 + r^2 g_{*}  , \ r \ge r_o , 
\end{align*}
where $g_{*}$ is  the standard metric  with area $4\pi$ on the sphere ${S}^2$.
Rewriting $\gamma_m$ as 
$ \gamma_m = ds^2 + u_m^2 (s) g_* , \ s \ge 0 $, 
one has $ u_m (0) = r_0 $ and  
\begin{equation} \label{eq-um-ODE}
 u_m'(s) = \left( 1 - \frac{2m}{u_m (s)} \right)^\frac{1}{2} .
\end{equation}
Let $ k > 0 $ be a  constant given by 
\be \label{eq-def-k}
k = \tau \left( 1 - \frac{2m}{r_o} \right)^{- \frac12} .
\ee
Define a metric 
\begin{align*}
\gamm =  A^2 dt^2+  r_o^{-2} u^2_m(A kt)  g(t).
\end{align*}
Here $ A > 0 $ is some constant which will be chosen later. 
The following properties of $(N, \gamm)$ follow from direct calculation (see (2.1) -- (2.16) in \cite{Miao-Xie16}):

\begin{itemize}
\item each $\Sigma_t : =\{ t  \} \times \Sigma$  has positive constant mean curvature w.r.t $ \p_t$;
\item  the induced metric on $\Sigma_0 $ is $g$, and the mean curvature of $\Sigma_0$ w.r.t $\p_t$ is  $H_o$;
\item  the Hawking mass of each $\Sigma_t$ is 
 \be \label{eq-mH-S1}
 \mh (\Sigma_t) = \frac12  \left( u_m ( A k t ) - r_o \right) ( 1 - k^2) + \mh (\Sigma) ;
 \ee
 \item
 the scalar curvature $ R(\gamm)$ of $\gamm$ satisfies 
\be \label{eq-R-gamm}
\begin{split}
R (\gamm) = & \  2 u_m^{-2} \left[ r_o^2 K_{ g(t) } - k^2 -  \frac18 | g' (t) |_{g(t)}^2  A^{-2}u_m^2  \right] \\
\ge & \ 2 u_m^{-2} \left[ \beta  - k^2 -   \frac12  \alpha  A^{-2} u_m^2 (A k )  \right] . 
\end{split}
\ee
\end{itemize} 

By \eqref{eq-R-gamm}, a sufficient condition to have $ R (\gamm) \ge 0 $ is that  
there exists  an $ A > 0 $ such that $  \beta  - k^2 -   \frac12  \alpha  A^{-2} u_m^2 (A k )  \ge 0 $.
If this is the case, then necessarily 
 $  k^2 < \beta \le 1  $. As  $ k^2 < 1 $ is equivalent to $ m < m_o$, 
where $ m_o  =  \frac{r_o}{2} ( 1 - \tau^2)$  is  the Hawking mass  of $ \Sigma$,
 such an $ A$ exists  only if $ m < m_o$.

\vspace{.3cm}

\noindent \underline{Step 2.} For any suitably  given $m < m_o$, we choose an optimal  $A = A_o$ 
such that $ \gamm$ has nonnegative scalar curvature. 

\begin{lemma} \label{lem-A-0}
For each $m \in ( - \infty, m_o) $ satisfying  
\be \label{eq-bak}
\beta > \left( 1 + \frac{\alpha}{2} \right) k^2, 
\ee
there exists a constant $A_o  > 0 $ such that
\be \label{eq-def-new-Ao}
\beta-k^2-\frac{\alpha}{2} A^{-2}_ou^2_m( A_o k )=0 .
\ee
Moreover, the set of all such $A_o$ is bounded from above and away from zero
as $ m $ tends to $ - \infty$. That is,  
there are constants $ B_2 > B_1 > 0 $ and $ \tilde m < 0 $  such that 
$
 B_1 < A_o < B_2 
$
whenever $ m < \tilde m  $.
\end{lemma}
\begin{proof}
Since $ \alpha > 0 $, \eqref{eq-def-new-Ao} is equivalent to  
\be
k^{-2} 2 \alpha^{-1} ( \beta-k^2 ) =  (A_o k) ^{-2}  u^2_m( A_o k ) . 
\ee
Consider the function $ f_m (s) =  s^{-1}  u_m (s)$. 
One has
$ \lim_{s \to 0 +} f_m (s) = \infty $ and 
\be
\lim_{s \to \infty } f_m (s) = \lim_{s \to \infty } u'_m (s) =  \lim_{s \to \infty} \left( 1 - \frac{2m}{u_m (s) } \right)^\frac12 = 1  .
\ee
Thus, the range of $ f_m $ includes $(1, \infty)$. 
Since \eqref{eq-bak} implies 
\bee
k^{-2} 2 \alpha^{-1} ( \beta-k^2 )  > 1 ,
\eee
the existence of such an $ A_o$ follows. 
 
Now, by  \eqref{eq-def-new-Ao} and the fact $u_m (s) \ge r_o $,  one has 
\be
\beta-k^2  = \frac{\alpha}{2} A^{-2}_ou^2_m( A_o k ) \ge  \frac{\alpha}{2} A^{-2}_o r_o^2 , 
\ee
which gives
\be \label{eq-lowerbd-Ao}
A_o^2   \ge  \frac{\alpha}{2}  r_o^2  \left( \beta-k^2 \right)^{-1} . 
\ee
As $ \lim_{ m \to - \infty} k = 0 $, this shows $ A_o$ is bounded away from $ 0 $ as $ m \to -  \infty$.

Next, suppose $ m < 0 $. By \eqref{eq-um-ODE},  $ u_m' (s) \le \left( 1 - \frac{2m}{r_o} \right)^\frac12  = \tau k^{-1}$, 
which implies 
$$
u_m (s) \le \ r_0 + \tau k^{-1} s  . 
$$
Thus, for $ 0 \le s \le A_o k $, 
\be
u_m'(s) \le \sqrt{ 1 - \frac{ 2m  ( r_0 + A_o \tau ) }{ u_m^2 (s) } } ,
\ee
or equivalently 
\be \label{eq-one-ode-um}
u_m (s) u_m'(s) \le \sqrt{ u_m^2 (s)  - 2m  ( r_0 +  A_o \tau ) } .
\ee
Upon integration, \eqref{eq-one-ode-um} shows 
\bee
u^2 _m (A_o k)  \le r_o^2 + A_o^2 k^2 + 2 A_o k \sqrt{ r_o^2 - 2m (r_o + A_o \tau) } .
\eee
Combined with \eqref{eq-def-new-Ao} and \eqref{eq-def-k}, this  implies 
\bee
 \beta-k^2  \le \frac{\alpha}{2}   A_o^{-2} 
\left( r_o^2 + A_o^2 k^2 + 2 A_o k \sqrt{ r_o^2 - 2m (r_o + A_o\tau) }  \right),
\eee
i.e. 
\be \label{eq-est-Ao-2}
 \beta-k^2 - \frac{\alpha}{2}    k^2   \le \frac{\alpha}{2}   
\left( r_o^2 A_o^{-2}   + 2 A_o^{-1}  \sqrt{ \tau^2  r_o^2 + ( \tau^2 - k^2) r_o A_o \tau }  \right).
\ee
Since $ \beta > 0 $ and $ \lim_{m \to - \infty} k = 0 $, it follows  from \eqref{eq-est-Ao-2} that 
$A_o$ is bounded from above as $ m \to - \infty$.
\end{proof}

In what follows, for each $ m $ satisfying \eqref{eq-bak}, we choose $ A $ to be the smallest root $A_o$ to 
equation \eqref{eq-def-new-Ao}. By \eqref{eq-R-gamm}, the  metric
$$ \gamm =  A_o^2 dt^2+  r_o^{-2} u^2_m(A_o kt) g(t) $$
has nonnegative scalar curvature. For each $m$, we glue $(N, \gamm)$ to $ \Omega$ 
by identifying $ \Sigma_0$ with $ \Sigma$. The argument in \cite[Section 3]{Miao-Xie16} leading to (3.9) therein
then gives
\be \label{eq-mh1}
\mh (\Sigma_1) \ge \sqrt{ \frac{ | \Sh |}{16 \pi} } , \ \ \mathrm{if} \ \Sh \neq \emptyset, 
\ee
 and 
\be \label{eq-mh2}
\mh (\Sigma_1) \ge  0, \ \ \mathrm{if}  \ \Sh = \emptyset .
\ee
Here, by \eqref{eq-mH-S1}, 
\be \label{eq-mH-S2}
\mh (\Sigma_1)  = \frac12  \left( u_m ( A_o k ) - r_o \right) ( 1 - k^2) + \mh (\Sigma) .
\ee

\vspace{.3cm}

\noindent \underline{Step 3.}  
We follow the idea in \cite{Miao-Xie16} by letting $ m \to - \infty$ in \eqref{eq-mh1} and \eqref{eq-mh2}.
Since $\lim_{m \to - \infty} k = 0$,  \eqref{eq-bak} is satisfied for every sufficiently negative $m$.  
By Lemma \ref{lem-A-0}, there exists a sequence $ \{ m_i \}$ with $ \lim_{ i \to \infty} m_i = - \infty$ 
such that the corresponding sequence $ \{ A_o^{(i)} \} $, where $ A_o^{(i)}$ is the $ A_o$ associated with
$ m_i$, has a finite limit. Consequently, by \eqref{eq-def-new-Ao}, 
the sequence $\{ u_{m_i} (A_o^{(i)} k^{(i)} ) \} $ has a finite limit as well.
Here $ k^{(i)}$ is the $k$ associated with $ m_i$.

We  evaluate $ \lim_{i \to \infty} u_{m_i} (A_o^{(i)} k^{(i)}  )$. One way to achieve this is to implicitly solve \eqref{eq-um-ODE}.
Suppose $ m < 0 $. Let $ v_m (s) > 0 $ be the function such that 
\be
 \frac{ - 2m}{ u_m (s)} = \sinh^{-2} ( v_m (s) ) .
\ee
In term of $ v_m (s)$, \eqref{eq-um-ODE} becomes
\begin{equation*}
-4m\sinh^2 (v_m (s))  v_m'(s) = 1 , 
\end{equation*}
or equivalently 
\begin{equation} \label{eq-vm-1}
(-m) [ \sinh ( 2 v_m (s) ) - 2 v_m (s) ]' = 1 .
\end{equation}
Plugging in 
$$
\sinh(2v_m (s) )
=  2  \left( \frac{ - u_m (s)}{2m} \right)^\frac12  \left(1 - \frac{ u_m (s) }{2m} \right)^\frac12 
$$
and 
$$v_m(s) = \ln\left(  \left( \frac{ - u_m (s)}{2m} \right)^\frac12  +   
\left(1 - \frac{ u_m (s) }{2m} \right)^\frac12 \right), $$ 
we get 
\bee
\begin{split}
& \ 2m  \left[ \ln  \left(  \left(   \frac{ - u_m (s)}{2m} \right)^\frac12 
+\left(1 - \frac{ u_m (s) }{2m} \right)^\frac12   \right) 
- \left(   \frac{ - u_m (s)}{2m} \right)^\frac12  \left(1 - \frac{ u_m (s) }{2m} \right)^\frac12   \right] \\
-  & \ 2m \left[  \ln\left(  \left(  -\frac{ r_o }{2m} \right)^\frac12
+ \left( 1-\frac{ r_o }{2m}  \right)^\frac12 \right)
- \left(  \frac{ - r_o }{2m} \right)^\frac12 \left( 1 - \frac{r_o }{2m} \right)^\frac12 \right]  = s . 
\end{split} 
\eee
Taking $ m = m_i$,  $ k = k^{(i)}$,  $ A_o = A_o^{(i)}$, $ s = A_o^{(i)} k^{(i)}$, and 
let $ u_{m_i}^{(i)} : = u_{m_i} ( A_o^{(i)}  k^{(i)} )$,  we have
\be \label{eq-umi-A-i}
\begin{split}
& \ 2m_i  \left[ \ln   \left(    \left( -\frac{ \umi }{2m_i} \right)^\frac12
+\left( 1-\frac{ \umi }{2m_i} \right)^\frac12  \right) 
-  \left(  \frac{ - \umi }{2m_i} \right)^\frac12 \left(  1 - \frac{ \umi }{2m_i} \right)^\frac12  \right] \\
-  & \ 2m_i \left[  \ln\left(   \left( -\frac{ r_o }{2m_i} \right)^\frac12 
+ \left( 1-\frac{ r_o }{2m_i} \right)^\frac12  \right)
- \left( \frac{ - r_o }{2m_i} \right)^\frac12  \left( 1 - \frac{r_o }{2m_i} \right)^\frac12  \right]  = A_o^{(i)} k^{(i)}  . 
\end{split} 
\ee
By Lemma \ref{lem-A-0}, $  \frac{ \umi }{2m_i}   = O ( | m_i |^{- 1} )  $  as $ i \to \infty$.
Hence, 
\bee
\begin{split}
& \  \ln\left( \left(   -\frac{ \umi }{2m_i} \right)^\frac12 
+ \left( 1-\frac{ \umi }{2m_i} \right)^\frac12    \right) 
-  \left( \frac{ - \umi }{2m_i} \right)^\frac12 \left( 1 - \frac{ \umi }{2m_i} \right)^\frac12  \\
= & \  - \frac23  \left( -\frac{ \umi }{2m_i} \right)^\frac32 + O ( | m_i |^{-2} ) .
\end{split}
\eee
Combined with \eqref{eq-def-k}, this gives 
\bee
\begin{split}
& \ \lim_{ i \to \infty}  \frac{2m_i}{k^{(i)}}  \left[ \ln \left( 
\left( -\frac{ \umi }{2m_i} \right)^\frac12 + 
\left( 1-\frac{ \umi }{2m_i} \right)^\frac12  \right) 
-  \left(  \frac{ - \umi }{2m_i} \right)^\frac12 \left(  1 - \frac{ \umi }{2m_i} \right)^\frac12  \right]  \\
= & \ \frac23   r_o^{-\frac12} \tau^{-1}   \lim_{ i \to \infty} {\umi}^\frac32  .
\end{split} 
\eee
Similarly, 
\bee
\lim_{ i \to  \infty}  \frac{2m_i}{k^{(i)}}  
\left[ \ln  \left(   \left( -\frac{ r_o }{2m_i} \right)^\frac12    
+ \left( 1-\frac{ r_o}{2m_i}  \right)^\frac12  \right) 
-  \left(  \frac{ - r_o }{2m_i} \right)^\frac12 \left( 1 - \frac{  r_o  }{2m_i} \right)^\frac12  \right]  \\
=  \frac23  r_o  \tau^{-1}  .
\eee
Hence, by \eqref{eq-umi-A-i}, we have
\be \label{eq-lim-eq-1}
 \lim_{ i \to \infty} {\umi} = r_o  \left( 1  + \frac32 \tau r_o^{-1} \lim_{ i \to \infty} A_o^{(i)}  \right)^\frac23 . 
\ee
Now let  $ \bar{A}_o : = \lim_{i \rightarrow \infty}A_o^{(i)} $. By \eqref{eq-lowerbd-Ao},  
$$ \bar{A}_o \ge \left( \frac{\alpha}{2\beta} \right)^\frac12 r_o > 0 .$$
Taking limit in  \eqref{eq-def-new-Ao}, we have
\be \label{eq-lim-eq-2}
\beta = \frac{\alpha}{2}  \bar{A}_o^{-2} \left(  \lim_{ i \to \infty} {\umi} \right)^2 .
\ee 
Therefore, it follows from  \eqref{eq-lim-eq-1} and \eqref{eq-lim-eq-2}
that 
\begin{equation} \label{eq-bar0}
\left(\frac{r_o}{\bar{A}_o}\right)^{\frac{3}{2}}
+\frac{3 \tau}{2}   \left(  \frac{r_o}{\bar{A}_o} \right)^\frac12 = \left(\frac{2\beta}{\alpha}\right)^{\frac{3}{4}}.
\end{equation}
We now  define $ \theta > 0 $ such that 
\be \label{eq-def-theta}
\frac{\bar{A}_o}{ r_o} = \theta^2 \left(\frac{\alpha} {2\beta} \right)^{\frac{1}{2}} .
\ee
Then  \eqref{eq-bar0} shows 
\be \label{eq-theta-e}
\theta^3 -  \frac{3 \tau}{2}  \left(\frac{\alpha} {2\beta} \right)^{\frac{1}{2}}  \theta^2  - 1 = 0 . 
\ee
By \eqref{eq-lim-eq-1}, \eqref{eq-def-theta} and \eqref{eq-theta-e}, we have 
\begin{equation}
\lim_{i \rightarrow \infty} \umi
 =  r_o \left( 1 +\frac{3  \tau}{2} \theta^2 \left(\frac{\alpha} {2\beta} \right)^{\frac{1}{2}}   \right)^{\frac{2}{3}} = r_o \theta^2 . 
\end{equation}
From this and \eqref{eq-mH-S2}, we conclude 
\be \label{eq-umi-limit}
\begin{split}
\lim_{i \to \infty} \m_{_H}(\Sigma_1) =  & \  
\lim_{i \to \infty}  \frac12  \left( \umi - r_o \right) ( 1 - {k^{(i)}}^2) + \mh (\Sigma)  \\
=  & \ 
\frac{r_o}{2} \left( \theta^2 - 1 \right)
+ \m_{_H}(\Sigma_o ) \\
=  & \ 
\frac{r_o}{2}\left( \theta^2 - \tau^2 \right) .
\end{split}
\ee
Here $ \mh(\Sigma_1)$ denotes the Hawking mass of $ \Sigma_1 $ in $(N, \gamma^{ (m_i) } )$.

\begin{remark}
Since $ \{ A_o^{(i)} \}$ can be taken to be any converging sequence, the argument above indeed shows 
$$
\lim_{m \rightarrow - \infty} u_m (A_o k ) = r_o \theta^2  
\ \ \mathrm{and} \  
\lim_{m \rightarrow - \infty} A_o = r_o \theta^2  \left(\frac{\alpha} {2\beta} \right)^{\frac{1}{2}} .
$$
\end{remark}

The following theorem follows directly from \eqref{eq-mh1}, \eqref{eq-mh2} and \eqref{eq-umi-limit}.

\begin{thm} \label{thm-fix-path}
Let $\Omega$, $ \Sigma$, $g$, $ r_o$, $H_o$ and $\tau$ be given in Theorem \ref{thm-main}. 
Let $ \{ g(t) \}_{t \in [0,1]}$ be a smooth path of metrics on $ \Sigma $ satisfying
$g(0)=g$, $g(1)$ is round,  $K_{g(t)} >  0 $  and 
$ \tr_{g(t)} g'(t)=0 $.  Let $ \alpha $ and $ \beta $ be 
the constants associated to $ \{ g(t) \}_{t \in [0,1]}$, given by \eqref{eq-alpha-beta}.
Let $ \theta > 0 $ be the number that is the unique root to  
\bee
\theta^3 -  \frac{3 \tau}{2}  \left(\frac{\alpha} {2\beta} \right)^{\frac{1}{2}}  \theta^2  - 1 = 0 . 
\eee
Then 
$$ \tau \le \theta \ \ \ \  \mathrm{if} \  \Sh = \emptyset , $$
and
$$ \tau^2 +  \frac{r_h}{r_o} \le  \theta^2    \ \ \ \  \mathrm{if} \  \Sh \neq \emptyset . $$
Here $ r_h = \sqrt{ \frac{ | \Sh |}{ 4\pi} } $ is the area radius of $ \Sh $.

\end{thm}

\begin{remark} \label{eq-statement-2}
Let $ f ( x ) = x ^3 -  \frac{3 \tau}{2}  \left(\frac{\alpha} {2\beta} \right)^{\frac{1}{2}}  x^2  - 1  $.
As $ f' ( x ) =  3 x  \left[ x   -  { \tau}  \left(\frac{\alpha} {2\beta} \right)^{\frac{1}{2}} \right]  $,
it is easily seen that, given a number $x$,
$$
x \le \theta \Longleftrightarrow f (x  ) \le 0 . 
$$
Thus, the conclusion in Theorem \ref{thm-fix-path} can be equivalently stated as
\be
 \tau^3  \left[ 1 -  \frac{3}{2}  \left(\frac{\alpha} {2\beta} \right)^{\frac{1}{2}} \right]   \le 1  \ \ \  \mathrm{if} \  \Sh = \emptyset ,
\ee
and
\be
\left( \tau^2 + \frac{r_h}{r_o} \right)^\frac32 -  \frac{3 \tau}{2}  \left(\frac{\alpha} {2\beta} \right)^{\frac{1}{2}}  
\left( \tau^2 + \frac{r_h}{r_o} \right) \le 1  \ \ \mathrm{if} \  \Sh \neq \emptyset .
\ee
\end{remark}

Part a) and b) of Theorem \ref{thm-main} now follow from Theorem \ref{thm-fix-path} by considering 
sequences of paths of metrics  $\{ g (t) \}_{ t \in [0,1]} $ with  $ \left(  \frac{\alpha}{ 2 \beta} \right)^\frac12 \to \zeta_g$.

\begin{remark}

Because we have chosen $A_o > 0 $ to be the smallest  root to \eqref{eq-def-new-Ao},
we have 
$ \beta - k^2 - \frac{\alpha}{2} A^{-2} u_m^2 (A k )  < 0 $, $\forall A \in (0, A_o)$. 
Thus, if  $ \tilde A_o  > 0 $ is any number such that 
$ \beta - k^2 - \frac{\alpha}{2} \tilde A_o^{-2} u_m^2 ( \tilde A_o k )  \ge 0 $,
we must have $ A_o \le \tilde A_o $, and hence 
$  u_m (A_o k ) \le u_m ( \tilde A_o k  ) $. 
Thus, besides requiring no assumptions on $\tau$, 
inequalities in a) and b) of Theorem \ref{thm-main}
are stronger than  those of Theorems 1.1 and 1.2  in \cite{Miao-Xie16}.
\end{remark}

\vspace{.2cm}

In the remaining part of this section, we prove part c) of Theorem \ref{thm-main}.
First, we review the definition of $\mb (\cdot)$.
Given a metric  $g$  and a function $H$ on a surface $ \Sigma$ that is
topologically a sphere, the Bartnik mass $\mb (\Sigma)$ associated to the triple $(\Sigma, g, H)$ 
\cite{Bartnik, BartnikTsingHua} can be defined as 
\begin{align*}
\inf\left\lbrace m_{_{ADM}} (M,\gamma)\,\vert\,(M,\gamma) \ \text{is an 
admissible extension of }(\Sigma, g, H)\right\rbrace . 
\end{align*}
Here $m_{_{ADM}}(\cdot)$ denotes the ADM mass \cite{ADM}, and 
an asymptotically flat $3$-manifold $(M,\gamma)$ with boundary  is  an admissible 
extension of $(\Sigma,g,H)$ 
if:  $\p M$  is isometric to $(\Sigma,g)$; 
 the mean curvature of $\p M$ in $(M, \gamma)$ equals $H$; 
$(M, g)$ has nonnegative scalar curvature; and either
 $(M,\gamma)$ contains no closed minimal surfaces (except possibly $\p M$), or 
$\p M$ is outer-minimizing in $(M,\gamma)$ (see \cite{Bray01, BC04, H-I01, Wiygul-16} for instance). 

Working with this definition, one sees that part c) of Theorem \ref{thm-main} 
would be a natural consequence of  the previous three steps. 
The reason is, because $\Sigma_1$ has a round intrinsic metric and constant mean curvature in $(N, \gamm)$, 
one can  attach part of a spatial Schwarzschild manifold with mass $ \mh (\Sigma_1)$,
outside a rotationally symmetric sphere isometric to $ \Sigma_1$, to $(N, \gamm)$ at $ \Sigma_1$.
The resulting manifold  would be an admissible extension of 
$(\Sigma, g, H_o)$, except it may not be smooth across $\Sigma_1$.
If it were smooth across $\Sigma_1$, then $ \mb (\Sigma) \le \mh(\Sigma_1)$ by definition. 
Passing to the limit in Step 3, one would obtain the estimate in  c).

To give a precise proof of c),  we can make use of a gluing result in \cite{CCMM}.
For this purpose, we return to the end of Step 2 to point out a few additional feature of $(N, \gamm)$.  
By \eqref{eq-mh1} and \eqref{eq-mh2},   the Hawing mass of $ \Sigma_1$ in $(N, \gamm)$  satisfies 
\be \label{eq-positive-mh}
\mh (\Sigma_1) \ge 0 . 
\ee
By   \eqref{eq-R-gamm}  and \eqref{eq-def-new-Ao}, the scalar curvature of $\gamm$ at any 
$ (x, t) \in \Sigma \times [0, 1)  \subset N$  satisfies 
\be \label{eq-R-gamm-p}
\begin{split}
R (\gamm) (x, t)  = & \  
2 u_m^{-2} ( A_o k t )  \left[ r_o^2 K_{ g(t) } (x)  - k^2 -  \frac18 | g' (t) |_{g(t)}^2  (x)  A_o^{-2}u_m^2  (A_o k t)  \right] \\
> & \ 2 u_m^{-2} (A_o k )  \left[ \beta  - k^2 -   \frac12  \alpha  A_o^{-2} u_m^2 (A k )  \right]  = 0 . 
\end{split}
\ee
At $ t = 1$, we also have 
\be \label{eq-R-gamm-p-1}
\begin{split}
R (\gamm) (x, 1)  = & \  
2 u_m^{-2} ( A_o k )  \left[ 1  - k^2 -  \frac18 | g' (1) |_{g(1)}^2  (x)  A_o^{-2}u_m^2  (A_o k )  \right] \\
> & \ 2 u_m^{-2} (A_o k )  \left[ \beta  - k^2 -   \frac12  \alpha  A_o^{-2} u_m^2 (A k )  \right]  = 0 ,
\end{split}
\ee
because $ \beta < 1 $ (since  $g(1)$ is round while $g(0) = g$ is not round).
Thus, $ R (\gamm) > 0 $ everywhere on $N$.

Now we can apply  \cite[Proposition 2.1]{CCMM} to $(N, \gamm)$. 
We may first assume  the path $\{ g(t) \}_{t \in [0,1]}$ has a property  
$g(t) = g(1)$ for $t$ in $(1- \delta, 1]$ for some $\delta>0$. In this case, 
a direct application of  \cite[Proposition 2.1]{CCMM} gives
\be \label{eq-after-CCMM}
\mb (\Sigma) \le \mh (\Sigma_1) . 
\ee
In general, by approximating $\{ g(t) \}_{t \in [0,1]}$ with paths satisfying such a property
(see (3.9) -- (3.13) in \cite{CCMM}), one knows \eqref{eq-after-CCMM} still holds.

Combining  \eqref{eq-umi-limit} and  \eqref{eq-after-CCMM}, we obtain
\be \label{eq-bmass-pf}
\begin{split}
\mb (\Sigma) \le & \  \lim_{i \to \infty} \m_{_H}(\Sigma_1) \\
=  & \  \frac{r_o}{2} \left( \theta^2 - 1 \right) + \m_{_H}(\Sigma_o ) .
\end{split}
\ee
Elementary estimates show that the root $\theta$ to \eqref{eq-theta-1} satisfies $ 1 \le  \theta \le 1 +  \frac{3}{2} \tau \zeta_g $.
Thus, 
\be
\mb (\Sigma) \le   \frac32 r_o  \left( 1 +  \frac34 \tau  \zeta_g \right) \tau \zeta_g   + \mh (\Sigma) .
\ee
This completes the proof of part c) of Theorem \ref{thm-main}.

\begin{remark} \label{eq-bmass-no-compact}
In Theorem \ref{thm-main}, we assume $\Sigma$ bounds a compact $3$-manifold with 
nonnegative scalar curvature. If this assumption is dropped, 
the above proof is still valid to show \eqref{eq-bmass-pf}, provided a sufficient condition  $ \mh (\Sigma) \ge 0  $
is assumed on $(g, H_o)$. This is because, by \eqref{eq-mH-S2},  $ \mh (\Sigma_1) > \mh (\Sigma) $  for each $(N, \gamm)$ 
used in the proof.
\end{remark}

\begin{remark}
In \cite{CCMM}, it was shown if  $(g, H_o) $ on $ \Sigma $ satisfies
$  \tau^2 < \frac{\beta}{1+\alpha} $, then
\be \label{eq-bmass-ccmm}
\mb(\Sigma) \le \left[ \frac{\alpha}{ \beta - ( 1 + \alpha )\tau^2  } \right]^\frac12 \tau \mh (\Sigma) + \mh (\Sigma) .  
\ee
Comparing \eqref{eq-bmass-pf} and \eqref{eq-bmass-ccmm}, we see
\eqref{eq-bmass-pf}  requires no assumptions on $ \tau$ and it improves
\eqref{eq-bmass-ccmm} when $\tau $ is small. For instance, as $ \tau \to 0$, 
\bee
\frac{\theta^2 - 1}{ 1 - \tau^2} = \left( \frac{\alpha}{2 \beta} \right)^\frac12 \tau + O (\tau^2) \ \ \mathrm{and} \ \ 
 \left[ \frac{\alpha}{ \beta - ( 1 + \alpha )\tau^2  } \right]^\frac12 \tau  =  \left( \frac{\alpha}{ \beta} \right)^\frac12 \tau  + O (\tau^2) .
\eee
\end{remark}

\section{Applications of Theorem \ref{thm-main}} \label{sec-other}
 
We apply Theorem \ref{thm-main} to prove Theorems  \ref{thm-positive-mH} and \ref{thm-positive-mH-b}.

\begin{lemma} \label{lem-ele}
Given two constants  $ b  > 0 $ and $ \lambda > 0 $, consider the function 
\be
\Phi (\tau) = \left( \tau^2 + \lambda  \right)^\frac32 -   b \tau 
\left( \tau^2 + \lambda  \right) - 1 , \ \ \tau \in (0, \infty). 
\ee
If  $ b < \min \{ \frac{\lambda}{ \sqrt{ 1 + \lambda} }, \frac{1}{\sqrt{ 1 + \lambda} } \}  $,  
 then $ \Phi (\tau) > 0 $ for any $ \tau \ge 1 $.
\end{lemma}

\begin{proof}
One has 
\be
\begin{split}
( 1 + \lambda )^{-1} \Phi (1 )  
= & \  \sqrt{ 1 + \lambda  }  - \frac{1}{ 1 + \lambda }  -    b  \\
= & \  \frac{\lambda} { (1 + \lambda) \left(  \sqrt{ 1 + \lambda } + 1 \right) }   + \frac{ \lambda } {\sqrt{ 1 + \lambda } } - b > 0 , 
\end{split}
\ee
and
\be
\begin{split}
 ( 3 \tau^2 + \lambda  )^{-1}  \Phi' (\tau) 
= & \ \frac{ 3  \left( \tau^2 + \lambda  \right)^\frac12  \tau  }{ 3 \tau^2 + \lambda } - b \\
\ge & \  \frac{  \tau  }{ \sqrt{  \tau^2 + \lambda } } - b > 0 
\end{split}
\ee
for $ \tau \ge 1 $. The lemma follows. 
\end{proof}

\begin{proof}[Proof of Theorem \ref{thm-positive-mH}]
We take the constant $ C = \frac{\sqrt{2}}{3}  $. Suppose 
\be \label{eq-small-zeta}
 \zeta_g < \frac{\sqrt{2}}{3}  \sqrt{ \frac{\p M}{ | \Sigma | } } . 
 \ee
Applying  b) of Theorem \ref{thm-main} to the domain $\Omega$, bounded by $\Sigma$ and $ \p M$, in $M$,
we have
\be \label{eq-theta-3}
\tau^2 + \frac{r_h}{r_o} \le \theta^2 ,
\ee
where $ r_o = \sqrt{ \frac{ | \Sigma | }{4\pi} } $, $r_h = \sqrt{ \frac{ | \p M |}{ 4 \pi} } $, 
$ \tau = \frac12 r_o H_o $, $H_o $ is the positive constant mean curvature of $\Sigma$,
and $\theta > 0 $ is the unique root to \eqref{eq-theta-1}. 
Similarly to Remark \ref{eq-statement-2}, we know \eqref{eq-theta-3} is equivalent to 
\be \label{eq-tau-penrose-final}
\left( \tau^2 + \frac{r_h}{r_o} \right)^\frac32 -  \frac{3 \tau \zeta_g}{2}   
\left( \tau^2 + \frac{r_h}{r_o} \right)  - 1 \le 0 .
\ee
Let  $ b = \frac{3 \zeta_g }{2}  $ and $ \lambda =  \frac{r_h}{r_o} $. 
Condition \eqref{eq-small-zeta}  becomes $ b < \frac{1}{ \sqrt{2} } \lambda $. 
Since $ | \p M | \le | \Sigma| $,  $ \lambda \le 1 $.
Thus, by \eqref{eq-tau-penrose-final} and Lemma \ref{lem-ele}, 
we have $ \tau < 1 $, i.e $ \mh (\Sigma) > 0$.
\end{proof}

Theorem \ref{thm-positive-mH-b} is proved in a similar way. 

\begin{proof}[Proof of Theorem \ref{thm-positive-mH-b}]
Since $ \Sigma $ bounds a compact $\Omega$ with nonnegative scalar curvature, 
$ \mb (\Sigma)$ satisfies the estimate in c) of Theorem \ref{thm-main}, i.e. 
\be
\tau^2 +  \frac{ 2 \mb (\Sigma) }{r_o}  \le \theta^2 .
\ee
Therefore, 
\be \label{eq-tau-bmass-final}
\left( \tau^2 + \frac{ 2 \mb (\Sigma) }{r_o} \right)^\frac32 -  \frac{3 \tau \zeta_g}{2}   
\left( \tau^2 + \frac{2 \mb(\Sigma) }{r_o} \right)  - 1 \le 0 .
\ee
Now suppose 
\be \label{eq-zeta-small-2}
 \zeta_g < \frac{\sqrt{2}}{3} \left(  1 +  \frac{2 \mb(\Sigma)}{r_o}  \right)^{-\frac12}  \min \left\{ \frac{2 \mb(\Sigma)}{r_o} , 1 \right\} .
 \ee 
Let $ b = \frac{3 \zeta_g }{2} $ and $ \lambda =  \frac{ 2 \mb (\Sigma) }{r_o}  $,  \eqref{eq-zeta-small-2} 
shows 
$$ 
b  < 
\left(  1 +  \lambda  \right)^{-\frac12}  \min \left\{ \lambda, 1 \right\} .
$$
By \eqref{eq-tau-bmass-final} and Lemma \ref{lem-ele}, 
we conclude $ \tau < 1 $, i.e $ \mh (\Sigma) > 0$.
\end{proof}

\section{Manifolds with negative scalar curvature} \label{sec-hyper}

In the remaining of this paper, we turn attention to CMC surfaces in  
manifolds with a negative lower bound on the scalar curvature.
Let $M$ denote a Riemannian $3$-manifold with scalar curvature $ R \ge - 6 \kappa^2$, where $ \kappa > 0 $ is a constant. 
Let $ \Sigma \subset M$ be a closed surface. 
In this context,  the hyperbolic Hawking mass of $ \Sigma$  is given by 
\be \label{eq-def-mh-h}
\mhh(\Sigma) = \sqrt{ \frac{ | \Sigma |}{16 \pi} } \left( 1 - \frac{1}{16 \pi} \int_\Sigma H^2 \, d \sigma + \frac{1}{4\pi} \kappa^2 | \Sigma | \right) .
\ee
A natural analogue of the Bartnik mass is
\begin{align*}
\mbh (\Sigma) =  \inf\left\lbrace  \m  (M^{_\mathbb{H}}, \gamma^{_{\mathbb{H} } })
\right\rbrace ,
\end{align*}
where $ \m (\cdot )$ is the mass of an asymptotically hyperbolic manifold and the infimum is taken over 
a space of  ``admissible asymptotically hyperbolic extensions" $(M^{_\mathbb{H} }, \gamma^{_{\mathbb{H} } } )$
of $(\Sigma, g , H)$. We refer readers to the recent work of Cabrera Pacheco,  Cederbaum and  McCormick \cite{CCM}
for a detailed discussion of this definition. 
 
We now let $\Omega$, $ \Sigma$, $g$, $H_o$, $r_o $ and $\tau$ be given in Theorem \ref{thm-intro-main-h}.
If $ g $ is a round metric, then $\zeta_g = 0 $ and $ \xi = 0 $. In this case, \eqref{eq-intro-main-mbh} reduces to 
$ \mbh (\Sigma) \le \mhh (\Sigma)$.
This is true because  a spatial AdS-Schwarzschild manifold with mass $\m = \mhh (\Sigma) $, lying outside a rotationally 
symmetric sphere isometric to $ \Sigma$, can be attached to $\Omega$ at $\Sigma$.

In what follows, we assume that $ g$ is not a round metric. 
Under the assumption $ K_g > - 3 \kappa^2$, there exists a smooth path of metrics $ \{ g(t) \}_{ 0 \le t \le 1} $ on $ \Sigma $ with 
$g(0)=g$, $g(1)$ is round,  $ K_{g(t)} >  - 3 \kappa^2 $, and $ \tr_{g(t)} g'(t)=0 $. 
(Existence of such a path can be provided by the solution to the normalized Ricci flow on $ \Sigma$ starting at $g$.
See \cite[Lemma 4.2]{Lin-Wang}  and \cite[Lemma 5.1]{CCM} for instance.)
We fix such a path $\{ g(t) \}_{ 0 \le t \le 1} $ and let $ \alpha$, $ \beta $ be the constants  given in \eqref{eq-alpha-beta}.
Then $ \alpha > 0 $ and $ 1 > \beta > -3 \kappa^2 r_o^2$. 

Similar to Step 1 in the proof of Theorem \ref{thm-main},  now one can consider a spatial AdS-Schwarzschild metric  
$
\gamma_m = \left( 1 - \frac{2 m}{r} + \kappa^2 r^2 \right)^{-1} d r^2 + r^2 g_*  , \ r \ge r_o ,
$
where $ m  $ is any parameter such that $ 1 - \frac{2 m}{r_o} + \kappa^2 r_o^2 > 0 $.
Rewriting $\gamma_m$ as 
$ \gamma_m = ds^2 + u_m^2 (s)  g_* , \ s \ge 0 $, 
one has $ u_m (0) = r_o $ and  
\begin{equation} \label{eq-um-ODE-h}
 u_m'(s) = \left( 1 - \frac{2m}{u_m (s)} + \kappa^2 u^2_m (s) \right)^\frac{1}{2} .
\end{equation}
Define a constant 
\be \label{eq-def-k-h}
k = \tau \left( 1 - \frac{2m}{r_o} + \kappa^2 r_o^2 \right)^{- \frac12} 
\ee
and a metric 
$$
\gamm =  A^2 dt^2+  r_o^{-2} u^2_m(A kt)  g(t)
$$
on  $ N = [0,1] \times \Sigma$, where $ A > 0 $ is a constant to be chosen. 
Direct calculation shows 
\begin{itemize}
\item each $\Sigma_t : =\{ t  \} \times \Sigma$  has positive constant mean curvature w.r.t $ \p_t$;
\item  the induced metric on $\Sigma_0 $ is $g$, and the mean curvature of $\Sigma_0$ w.r.t $\p_t$ is  $H_o$;
\item  the hyperbolic Hawking mass of each $\Sigma_t$ is 
 \be \label{eq-mH-S1-h}
 \begin{split}
 \mhh (\Sigma_t) = & \  \frac12  \left( u_m ( A k t ) - r_o \right) ( 1 - k^2) \\
& \ + \frac12 \kappa^2 ( 1 - k^2) ( u_m^3 (Akt)  - r_o^3 )  + \mhh (\Sigma) ;
 \end{split}
 \ee
 \item
 the scalar curvature $ R(\gamm)$ of $\gamm$ satisfies 
\be \label{eq-R-gamm-h}
\begin{split}
R (\gamm) = & \   2 u_m^{-2} ( r_o^2 K_{ g(t) } - k^2 )  -  \frac14 | g' (t) |_{g(t)}^2  A^{-2}  - 6  k^2 \kappa^2 \\
 \ge  & \   2 u_m^{-2} ( \beta  - k^2 )  -  \alpha  A^{-2}  - 6  k^2 \kappa^2 . 
\end{split}
\ee
\end{itemize} 

\begin{remark}
The manifold $(N, \gamm)$, constructed above via the warping function of a AdS-Schwarzschild metric,  
was also used  in \cite{CCM}. Estimates on $\mbh(\cdot)$ for a pair $(g, H_o)$ were derived in \cite{CCM}
under suitable smallness conditions on $H_o$. 
In this section, by assuming $(g, H_o)$ arises from the boundary of $ \Omega$
and by making an optimal choice of $A$, we obtain estimates on $\mbh( \cdot )$ that 
require  no assumption on $H_o$. 
\end{remark}

By \eqref{eq-R-gamm-h}, a sufficient condition 
to guarantee $ R (\gamm) \ge - 6 \kappa^2 $ on $ N $ is
\be \label{eq-R-transit}
 u_m^{-2} (Ak t ) ( \beta  - k^2 )  -  \frac{1}{2} \alpha  A^{-2} + 3 \kappa^2 ( 1 - k^2)  \ge 0 , \ \ \forall \ t \in [0,1] .
\ee
As  $ u_m (s)$ is monotone, \eqref{eq-R-transit} is equivalent to 
\be  \label{eq-Ao-beta-p-0} 
 u_m^{-2} (Ak ) ( \beta  - k^2 )  -  \frac12 \alpha  A^{-2}  + 3 \kappa^2 (1-k^2)  \ge 0,   \ \ \mathrm{if} \ \beta - k^2 > 0 ,
\ee
or 
\be  \label{eq-Ao-beta-n-0} 
 r_o^{-2}  ( \beta  - k^2 )  -  \frac12 \alpha  A^{-2}  + 3 \kappa^2 (1-k^2)  \ge 0,   \ \ \mathrm{if} \ \beta - k^2 \le 0 .
\ee

Next, as in Step 2 in the proof of Theorem \ref{thm-main}, we choose an optimal $ A = A_o $ 
so that \eqref{eq-Ao-beta-p-0} or \eqref{eq-Ao-beta-n-0}  are met. 
If  $ \beta \le 0 $,  using the fact  $ \beta +  3 \kappa^2 r_o^2 > 0 $,
one easily sees  an optimal $A$ satisfying \eqref{eq-Ao-beta-n-0} is 
\be \label{eq-Ao-beta-n}
A_o =  r_o \left( \frac{ \frac12 \alpha}{ \beta + 3 \kappa^2 r_o^2   - \left(  1 + 3 \kappa^2 r_o^2 \right) k^2 } \right)^{\frac12} ,
\ee
provided  $ k $ is small. 

If  $ \beta > 0 $ (which occurs only if $ \inf_\Sigma K_g > 0 $),
we choose an optimal $A_o $ satisfying  \eqref{eq-Ao-beta-p-0}  according to the following lemma.

\begin{lemma} \label{lem-Ao-h}
Suppose $ \alpha > 0 $ and $ \beta > 0 $. For  every $m \in ( - \infty, 0) $ satisfying $ k^2 < \beta$, 
there exists a  positive constant $A_o  $ such that
\be \label{eq-def-new-Ao-h}
 ( \beta  - k^2 )  + \left[ 3 \kappa^2  ( 1 - k^2)    -  \frac12 \alpha  A^{-2}   \right]  u_m^2 (A_o k) = 0    .
\ee
Moreover, the set of all such $A_o$ is bounded from above and away from zero as $ m$ tends to $ - \infty$.
\end{lemma}
\begin{proof} 
For each fixed $m$, consider the function 
$$  f_m (A ) =  ( \beta  - k^2 )  + \left[ 3 \kappa^2 (1-k^2) -  \frac12 \alpha  A^{-2}   \right]  u_m^2 (A k) , \ 
A \in (0, \infty) .  $$
One has 
$ \lim_{A  \to 0 +} f_m (A ) =  - \infty $ since   $ \alpha > 0 $, and 
$  \lim_{A \to \infty } f_m (A ) = \infty $  because  $\lim_{s \to \infty} u_m (s) = \infty$ and  $ k^2 < \beta < 1 $.
Moreover, $ f_m (A)$ is strictly increasing in $ A$.
Hence, there exists a unique root $A_o > 0 $ to \eqref{eq-def-new-Ao-h}. 
For this $ A_o$, one has 
\be \label{eq-new-A-bd-h}
 3 \kappa^2 (1 - k^2)      -  \frac12 \alpha  A_o^{-2} \le 0 ,
\ee
for otherwise the left side of \eqref{eq-def-new-Ao-h} would be positive. Thus, 
\be
A_o^2 \le  \frac16 \alpha \kappa^{-2} (1 - k^2)^{-1} . 
\ee
As $ \lim_{m \to -\infty} k = 0 $, this shows that $ A_o$ is bounded from above as $ m $ tends to $ - \infty$. 
On the other hand, by \eqref{eq-def-new-Ao-h}, \eqref{eq-new-A-bd-h} and  the fact $ u_m (s) \ge r_o $, one has
$$
0 \le  ( \beta  - k^2 )  + \left[ 3 \kappa^2  ( 1 - k^2)    -   \frac12 \alpha  A_o^{-2}   \right]  r_o^2 ,
$$
i.e.
\be \label{eq-lower-bd-Ao-h}
\alpha   r_o^2 \left[  2 ( \beta - k^2 )  + 6 \kappa^2  ( 1 - k^2) r_o^2  \right]^{-1}
 \le  A_o^{2} .
\ee
This shows $ A_o $ is bounded away from $0$ as $ m$  tends to $ - \infty$.
\end{proof}

In what follows, we assume $ m $ is sufficiently negatively large  so that $ k^2$ is small. 
We choose $ A = A_o > 0$ so that $ A_o$ is the unique root to \eqref{eq-def-new-Ao-h} if $ \beta > 0 $; and
$ A_o $ is given by \eqref{eq-Ao-beta-n} if $ \beta \le 0$.
In either case, $ A_o = O (1) $,  as $ m \to - \infty$.

Before we compute  $\lim_{ m \to - \infty} A_o $ and $\lim_{ m \to - \infty} u_m (A_o k) $, we 
point out the non-negativity of $ \mhh (\Sigma_1) $ in our setting. 

\begin{prop} \label{prop-mhs-h}
Let $\Omega$, $ \Sigma$, $(N, \gamm) $, $A_o$, be given above. Then 
$ \mhh (\Sigma_1) \ge 0 $.
\end{prop}

\begin{proof}
This is essentially a consequence of the positive mass theorem on asymptotically hyperbolic manifolds 
(see  \cite{C-M03, WangX} for instance).
More precisely, this follows from such a theorem on manifolds 
with corners along a hypersurface (see \cite{Bonini-Qing} and also \cite{WangYau06, ShiTam06}).
Consider three manifolds 
$$ \Omega, (N, \gamm), \ \mathrm{and} \ M_{_{\m}} ,$$
where $ M_{_{\m}} $ is part of the spatial AdS-Schwarzschild manifold,  with mass $ \m = \mhh (\Sigma_1)$,
lying  outside a rotationally symmetric sphere $S$ isometric to $ \Sigma_1$. 
One can glue $M_{_{\m}}$ to $(N, \gamm)$  by identifying $ S $ with $\Sigma_1$ and 
glue $\Omega$ to $(N, \gamm)$ by identifying $\Sigma$ with $ \Sigma_0$. 
Applying \cite[Theorem 1.1]{Bonini-Qing} to the resulting manifold, one concludes 
$ \m  \ge 0 $. 
\end{proof}

The above proof of Proposition \ref{prop-mhs-h} indeed indicates $  \mbh (\Sigma) \le \mhh (\Sigma_1)  $ 
if the manifold obtained by gluing $M_{\m}$ and $(N, \gamm)$ along $ \Sigma_1$ is smooth.
By invoking a gluing result in \cite[Proposition 3.3]{CCM}, one can verify this assertion.

\begin{prop} \label{prop-mbh-m}
Let $\Omega$, $ \Sigma$, $(N, \gamm) $, $A_o$, be given above. Then $ \mbh (\Sigma) \le \mhh(\Sigma_1)$.
\end{prop}
\begin{proof}
Since $ \beta < 1 $ and $ \Sigma_1$ is round in $(N, \gamm)$, an examination of \eqref{eq-Ao-beta-p-0}  and \eqref{eq-Ao-beta-n-0}  
shows $ R ( \gamm ) > - 6 \kappa^2 $ near $ \Sigma_1$ in $N $.
The claim nows follows from Proposition \ref{prop-mhs-h} and \cite[Proposition 3.3]{CCM} in the same way that 
\eqref{eq-after-CCMM} follows from \eqref{eq-positive-mh} and \cite[Proposition 2.1]{CCMM}.
\end{proof}

Next, we proceed to evaluate $\lim_{ m \to - \infty} A_o $ and $\lim_{ m \to - \infty} u_m (A_o k) $. 
First, as  $ m < 0 $, \eqref{eq-um-ODE-h} implies 
$$
u_m'(s) \le \left( 1 - \frac{2m}{r_o} + \kappa^2 u_m^2 (s) \right)^\frac12 ,
$$
which, upon integration, gives 
\be
\kappa u_m (A_o k ) + \sqrt{ 1 - \frac{2m}{r_o} + \kappa^2 u_m^2 (A_o k)  }
\le e^{\kappa A_o k } \left[ \kappa r_o + \sqrt{ 1 - \frac{2m}{r_o} + \kappa^2  r_o^2  } \right].
\ee
This yields 
\be
 u_m (A_o k ) \le u_m^*,
 \ee
 where
 \be \label{eq-def-um*}
 \begin{split}
 u_m^* = & \ 
 \frac{ e^{\kappa A_o k } \left( \kappa r_o + \sqrt{ 1 - \frac{2m}{r_o} + \kappa^2  r_o^2  } \right)^2 
 - e^{ - \kappa A_o k } \left(  1 - \frac{2m}{r_o}  \right)  }{ 2 \kappa  \left( \kappa r_o + \sqrt{ 1 - \frac{2m}{r_o} + \kappa^2  r_o^2  } \right)} \\
 = & \ r_o \left[ \frac12 \left( e^{ \kappa A_o k } + e^{ - \kappa A_o k } \right) + \frac{1}{2 \kappa k} \left( e^{ \kappa A_o k }- 
 e^{ - \kappa A_o k } \right)  \frac{H_o}{2}  \right].
\end{split} 
 \ee 
For $  0 \le s  \le A_o k $, by \eqref{eq-um-ODE-h}, we have
\be
\begin{split}
u_m'(s) \le & \  \left( \frac{ u_m^*  - {2m} + \kappa^2 {u_m^*}^3 }{ u_m (s) }\right)^\frac12 ,
\end{split}
\ee
which implies 
\be \label{eq-upper-um-h}
u_m^\frac32 (A_o k )  \le \frac32  A_o k  \sqrt{ u_m^*  - {2m} + \kappa^2 {u_m^*}^3 }  + r_o^\frac32 . 
\ee
On the other hand, by \eqref{eq-um-ODE-h}, 
\be
u_m'(s) \ge   \left(  - \frac{2m} {u_m (s) }  \right)^\frac12 ,
\ee
which implies
\be  \label{eq-upper-um-h-2}
u_m^\frac32 (A_o k ) \ge \frac32  A_o k  \sqrt{  - {2m} } + r_o^\frac32 . 
\ee

Now, let $ \{ m_i \}$ denote any sequence that tends to $- \infty$ so that 
the corresponding sequence $ \{ A_o^{(i)} \} $ has a finite limit,  where $ A_o^{(i)}$ is the $ A_o$ associated with
$ m_i$. Let $ \bar{A}_o : = \lim_{ m \to - \infty} A_o^{(i)} $.  As $ \lim_{i \to \infty} k = 0 $, 
by  \eqref{eq-def-k-h}  and \eqref{eq-def-um*}, 
\be
 \lim_{i \to \infty} u_{m_i}^*  =  r_o \left( 1 + \frac{1}{2}   \bar{A}_o H_o \right) 
\ee
and 
\be
\lim_{  i \to  \infty} k  \sqrt{ - 2m_i }  =  \tau r_o^\frac12  = 
\lim_{  i \to  \infty}  k \sqrt{ u_{m_i}^*  - {2m_i} + \kappa^2 {u_{m_i}^*}^3 } .
\ee
Hence, if we let 
 $$ \xi  = \bar{A}_o r_o^{-1}, $$ 
then, by \eqref{eq-upper-um-h}  and \eqref{eq-upper-um-h-2}, 
\be \label{eq-umi-limit-n}
\bar{u}_o : =  \lim_{i \to \infty} u_{m_i} ( A_o^{(i)} k )  = r_o  \left( 1 +  \frac32  \tau \xi  \right)^\frac23 .
\ee

As a result of  \eqref{eq-umi-limit-n} and \eqref{eq-mH-S1-h}, 
we see that the limit of the hyperbolic Hawking mass of $ \Sigma_1$ in $(N, \gamma^{(m_i)})$ is given by
 \be \label{eq-mH-S1-h-more-0}
 \begin{split}
\lim_{i \to \infty} \mhh (\Sigma_1) = & \  \frac12  \left( \bar{u}_o - r_o \right) 
 + \frac12 \kappa^2  ( \bar{u}_o^3  - r_o^3 )  + \mhh (\Sigma) \\
 = & \ \frac{r_o}{2} \left[ \left( 1 +  \frac32  \tau \xi  \right)^\frac23   +  \kappa^2 r_o^2   \left( 1 +  \frac32  \tau \xi  \right)^2  - 1- \kappa^2 r_o^2   \right] + \mhh(\Sigma) .
 \end{split}
 \ee
Here,  by \eqref{eq-Ao-beta-n},
\be \label{eq-Ao-beta-n-more}
\xi  =   \left( \frac{ \frac12 \alpha}{  \beta + 3 \kappa^2 r_o^2  } \right)^{\frac12} , \ \ \mathrm{if} \ \beta \le 0  . 
\ee
When $\beta > 0 $, by \eqref{eq-lower-bd-Ao-h},
\be  \label{eq-low-bd-xi}
   \xi  \ge \left(  \frac{ \frac12 \alpha }{   \beta  + 3 \kappa^2   r_o^2   }  \right)^\frac12 > 0 .
\ee
Hence, by \eqref{eq-def-new-Ao-h} and \eqref{eq-umi-limit-n}, 
\be \label{eq-def-new-Ao-h-more}
 \beta +  \left( 3 \kappa^2  r_o^2 -  \frac{\alpha}{2}  \xi^{-2}   \right)  \left( 1 +  \frac32  \tau \xi   \right)^\frac43 = 0,
\ee
or equivalently 
\be \label{eq-def-new-Ao-h-more-g}
\left[ \beta +  3 \kappa^2  r_o^2  \left( 1 +  \frac32  \tau \xi   \right)^\frac43 \right] \xi^2  -  \frac{\alpha}{2}   \left( 1 +  \frac32  \tau \xi   \right)^\frac43   = 0   .
\ee

\begin{remark} \label{rem-one-root}
Consider 
$$ \Psi ( x ) =  \beta +  \left( 3 \kappa^2  r_o^2 -  \frac{\alpha}{2}  x^{-2}   \right)  \left( 1 +  \frac32  \tau x   \right)^\frac43 , 
 \  x \in (0, \infty) . $$
Then 
$$ \Psi'(x) = \left(   6 \kappa^2 r_o^2 \tau +   \alpha x^{-3} + \frac12 \alpha \tau x^{-2}  \right) 
\left( 1 +  \frac32  \tau x   \right)^\frac13  > 0 .  $$ 
As  $ \lim_{x \to 0+ } \Psi (x )  = - \infty$ and  $ \lim_{x  \to \infty} \Psi ( x ) = \infty$, 
$\Psi (x) $ has a unique root $ \xi > 0 $.
\end{remark}

\begin{remark} 
Since $ \{ A_o^{(i)} \}$ can be any converging sequence, the argument above  shows 
$$
\lim_{m \rightarrow - \infty} A_o = r_o \xi
\ \ \mathrm{and} \  
\lim_{m \rightarrow - \infty} u_m (A_o k ) = r_o   \left( 1 +  \frac32  \tau \xi  \right)^\frac23 .
$$
\end{remark}

Suppose  $ \beta > 0 $, we want to estimate $ \xi > 0 $ which is the solution to \eqref{eq-def-new-Ao-h-more-g}. 
Similar to \eqref{eq-def-theta}, we make a change of variable by letting 
$$ \xi =  \left( \frac{\alpha}{2 \beta} \right)^\frac12  \theta_{\kappa}^2 $$  
for $ \theta_\kappa > 0 $. Then \eqref{eq-def-new-Ao-h-more-g} becomes 
\be \label{eq-theta-h}
\left[ 1 +  3 \kappa^2  r_o^2 \beta^{-1} \left( 1 +  \frac32  \left( \frac{\alpha}{2 \beta} \right)^\frac12   \tau   \theta_\kappa^2 
   \right)^\frac43 \right]^\frac34  \theta_\kappa^3   
-  \frac32    \left( \frac{\alpha}{2 \beta}  \right)^\frac12  \tau  \theta_\kappa^2 - 1 = 0   . 
\ee 
For $ x \in (0, \infty)$, consider the function 
\be
f (x) =  \left[ 1 +  3 \kappa^2  r_o^2 \beta^{-1} \left( 1 +  \frac32  \left( \frac{\alpha}{2 \beta} \right)^\frac12   \tau   x^2 
   \right)^\frac43 \right]^\frac34  x^3   
-  \frac32    \left( \frac{\alpha}{2 \beta}  \right)^\frac12  \tau  x^2 - 1   .
\ee
By Remark \ref{rem-one-root}, $ f (x)$  has a unique positive root $\theta_\kappa$.
As in Theorem \ref{thm-fix-path}, we let $ \theta > 0 $ be the unique root to 
\be
\theta^3  -  \frac32    \left( \frac{\alpha}{2 \beta}  \right)^\frac12  \tau  \theta^2 - 1    = 0 .
\ee
Then $ f ( \theta ) \ge 0 $.  Therefore, we conclude
\be \label{eq-theta-theta-kappa}
\theta_\kappa  \le \theta . 
\ee
In particular, using the fact  $ \theta \le 1 +  \frac32    \left( \frac{\alpha}{2 \beta}  \right)^\frac12  \tau $, we have
\be  \label{eq-est-xi-h-h}
\theta_{\kappa}  \le  1 +  \frac32    \left( \frac{\alpha}{2 \beta}  \right)^\frac12  \tau .
\ee

Our above discussion has established the following theorem. 

\begin{thm} \label{thm-main-h}
Let $\Omega$, $ \Sigma$, $g$, $H_o$, $r_o $ and $\tau$ be given in Theorem \ref{thm-intro-main-h}.
Suppose  $g$ is not a round metric and its Gauss curvature satisfies $K_g > - 3 \kappa^2$.
Let $ \{ g(t) \}_{ 0 \le t \le 1} $ be a given smooth path of metrics  on $ \Sigma $ satisfying 
$g(0)=g$, $g(1)$ is round,  $ K_{g(t)} >  - 3 \kappa^2  $,  and 
$ \tr_{g(t)} g'(t)=0 $. 
Then the hyperbolic Hawking mass $\mbh (\Sigma)$ satisfies 
\be  \label{eq-main-mbh}
\begin{split}
& \  \mb^{\mathbb{H}}(\Sigma) - \mhh(\Sigma) \\
\le & \  \frac{r_o}{2} \left[  \kappa^2 r_o^2  \left( 1 +  \frac32  \tau \xi  \right)^2  + 
  \left( 1 +  \frac32  \tau \xi  \right)^\frac23  - \kappa^2 r_o^2 - 1    \right] .
\end{split}
\ee
Here   $ \xi > 0 $ is a constant given by 
\be \label{eq-xi-beta-n-h}
\xi =   \left( \frac{ \frac12 \alpha}{  \beta + 3 \kappa^2 r_o^2  } \right)^{\frac12} , \ \ \mathrm{if} \ \beta \le 0  ;
\ee
and  $\xi $ is the unique positive  root to 
\be \label{eq-def-new-Ao-h-more-g-h}
\left[ \beta +  3 \kappa^2  r_o^2  \left( 1 +  \frac32  \tau \xi   \right)^\frac43 \right] \xi^2  -  \frac{\alpha}{2}   \left( 1 +  \frac32  \tau \xi   \right)^\frac43   = 0   , \ \ \mathrm{if} \ \beta > 0.
\ee
In the latter case,  if one writes $\xi  = \left( \frac{ \alpha}{ 2 \beta} \right)^\frac12 \theta_\kappa^2 $ for a positive  $ \theta_{\kappa}  $, 
then $ \theta_\kappa \le \theta $ where $\theta > 0 $ is the unique root to 
$$ \theta^3  -  \frac32    \left( \frac{\alpha}{2 \beta}  \right)^\frac12  \tau  \theta^2 - 1    = 0 . $$
In particular, this shows
\be \label{eq-est-xi-h-h-h}
\xi \le  \left( \frac{\alpha}{2 \beta} \right)^\frac12  \left[  1 +  \frac32    \left( \frac{\alpha}{2 \beta}  \right)^\frac12  \tau \right]^2 . 
\ee
\end{thm}

Theorem \ref{thm-intro-main-h} is a corollary of Theorem \ref{thm-main-h}.

\begin{proof}[Proof of Theorem \ref{thm-intro-main-h}]
Note that the second inequality in \eqref{eq-intro-main-mbh} simply follows from 
\bee \label{eq-final}
\begin{split}
& \  \kappa^2 r_o^2  \left( 1 +  \frac32  \tau x \right)^2  + 
  \left( 1 +  \frac32  \tau x \right)^\frac23  - \kappa^2 r_o^2 - 1  \\
= & \   \tau x   \left( 1 +  \frac34  \tau x   \right) \left[  3 \kappa^2 r_o^2 + 
\frac{ 3   } {  \left( 1 +  \frac32  \tau x  \right)^\frac43   +  \left( 1 +  \frac32  \tau x   \right)^\frac23  + 1  }  \right] , \ x \ge 0. 
\end{split}
\eee

If  $\inf_\Sigma K_g \le 0 $,  the pair $(\alpha, \beta)$ associated to any path 
$ \{ g(t) \}_{ 0 \le t \le 1} $ with  $g(0)=g$, $g(1)$ is round,  $ K_{g(t)} >  - 3 \kappa^2  $,  and 
$ \tr_{g(t)} g'(t)=0 $, necessarily has  $ \beta \le 0 $. Thus, (i) follows from taking the 
infimum over such paths in  \eqref{eq-xi-beta-n-h} of Theorem \ref{thm-main-h}.

Suppose $\inf_\Sigma K_g > 0 $, moreover we assume $g$ is not a round metric. 
In this case, we can restrict the attention to the paths $ \{ g(t) \}_{ 0 \le t \le 1} $ with  $g(0)=g$, $g(1)$ is round,  $ K_{g(t)} > 0$,  and 
$ \tr_{g(t)} g'(t)=0 $. A pair $(\alpha, \beta)$ associated to such a path has $\beta > 0 $. 
Applying Theorem \ref{thm-main-h}, by \eqref{eq-def-new-Ao-h-more-g-h}, 
we see \eqref{eq-intro-main-mbh} holds for 
\be \label{eq-sec-inf-root}
\xi = \inf_{ \{ g(t ) \}  }  \left\{\mathrm{the \  root \ of} \left[ \beta +  3 \kappa^2  r_o^2  \left( 1 +  \frac32  \tau x   \right)^\frac43 \right] x^2  -  \frac{\alpha}{2}   \left( 1 +  \frac32  \tau x   \right)^\frac43   = 0   \right\} .
\ee
Furthermore, such an $\xi$ satisfies 
$$ \xi \le \zeta_g  \theta^2 ,$$ 
where $\theta$ is the unique root to 
$ \theta^3 - \frac32 \zeta_g \tau \theta^2 - 1 = 0 $.
Since $ \theta \le 1 + \frac32 \zeta_g \tau $, we have  $  \xi \le \zeta_g \left(  1 + \frac32 \zeta_g \tau \right)^2$.
This completes the proof. 
\end{proof}

We end this paper with a remark that discusses the analogues of \eqref{eq-mh-bound}.

\begin{remark}
In the context of Theorem \ref{thm-intro-main-h}, one indeed has 
\be  \label{eq-end-main-mbh}
\begin{split}
  \mhh(\Sigma) +   \frac{r_o}{2} \left[  \kappa^2 r_o^2  \left( 1 +  \frac32  \tau \xi  \right)^2  + 
  \left( 1 +  \frac32  \tau \xi  \right)^\frac23  - \kappa^2 r_o^2 - 1    \right] \ge 0 .
\end{split}
\ee
This follows from Proposition \ref{prop-mhs-h} and \eqref{eq-mH-S1-h-more-0}. Recall that
Proposition \ref{prop-mhs-h}  is a consequence of the positive mass theorem on asymptotically hyperbolic manifolds.

Next suppose the compact manifold $ \Omega$ in Theorem \ref{thm-intro-main-h} has an additional CMC boundary component $ \Sh : = \p \Omega \setminus \Sigma$ whose mean curvature equals $2 \kappa$ (with respect to the inward normal). Suppose $\Sh$ minimizes area among surfaces enclosing $\Sh$ in $\Omega$.
Assuming the Penrose inequality on asymptotically hyperbolic manifolds holds valid (see \cite{WangX, Neves07} for a statement of this conjecture), one would have
$ \mhh(\Sigma_1) \ge \frac{1}{2}  r_h$ 
as a replacement of  Proposition \ref{prop-mhs-h}. Here $ r_h  $ is the area radius of $\Sh$.
This combined with \eqref{eq-mH-S1-h-more-0} then implies 
\be \label{eq-test-HPenrose}
 \mhh(\Sigma) +   \frac{r_o}{2} \left[  \kappa^2 r_o^2  \left( 1 +  \frac32  \tau \xi  \right)^2  + 
  \left( 1 +  \frac32  \tau \xi  \right)^\frac23  - \kappa^2 r_o^2 - 1    \right] \ge \frac{r_h}{2} .
\ee
Since the hyperbolic Penrose inequality is still open, the above inequality 
on compact manifolds $\Omega$ may serve as a test of the hyperbolic Penrose inequality. 
\end{remark}

\end{document}